\documentclass[12pt]{amsart}
\usepackage[utf8]{inputenc}
\usepackage[russian,english]{babel}
\usepackage{amscd,amssymb, latexsym}
\usepackage{amsmath,amsthm,amstext, amsbsy, amssymb,latexsym}
\usepackage{epsfig,graphics,graphicx,mathrsfs,color}
\usepackage{a4wide}
\usepackage{ulem}
\usepackage{hyperref}
\newtheorem{theorem}{Theorem}
\newtheorem{corollary}{Corollary}
\newtheorem{proposition}{Proposition}

\theoremstyle{remark}

\makeatletter
\@addtoreset{equation}{section}
\makeatother

\begin{document}

\title[On Koebe radius...]{On Koebe radius and coefficients estimate for univalent harmonic mappings}

\author[M.~Borovikov]{Mikhail Borovikov}

\begin{abstract}
    The problem on estimating of the Koebe radius for univalent harmonic mappings
    of the unit disk $\mathbb D=\{z\in\mathbb C : |z|<1\}$ is considered. For a
    subclass of harmonic mappings with the standard normalization and a certain
    growth estimate for analytic dilatation, we provide a new estimate for the
    Koebe radius. A new estimate for Taylor coefficients of the holomorphic part of
    a function from the subclass under consideration is obtained as a corollary.
\end{abstract}

\thanks{The study was carried out with the financial support of the Ministry of Science and Higher Education of the Russian Federation in the framework of a scientific project under agreement No. 075-15-2025-013. 
Author is a member of the group “Leader” followed by Alexander Bufetov
which won the contest conducted by
Foundation for the Advancement of Theoretical Physics and Mathematics “BASIS”
and would like to thank its sponsors and jury.}

\address{
\hskip -\parindent Mikhail Borovikov${}^{1,2,3}$:
\newline \indent 1)~Faculty of Mechanics and Mathematics,
\newline \indent Lomonosov Moscow State University, Moscow, Russia;
\newline \indent 2)~Moscow Center for Fundamental and Applied Mathematics,
\newline \indent Lomonosov Moscow State University, Moscow, Russia;
\newline \indent 3)~St.~Petersburg State University, St.~Petersburg, Russia;
\newline \indent {\tt misha.borovikov@gmail.com}
}

\maketitle

\section{Introduction}\label{sec1}

The paper is devoted to the problem on estimate of the Koebe radius for
univalent harmonic mappings of the unit disk $\mathbb D=\{z\in\mathbb C :
|z|<1\}$ and for coefficients estimate problem for such mappings. For a certain
subclass of the class $\mathcal S_H^0$ of harmonic mappings with the standard
normalization we provide new estimate for the Koebe radius. Using this result
we present new estimates for Taylor coefficients of the holomorphic part of a
function from the subclass under consideration.

Let us explain what we mean by the \emph{Koebe radius} for a class of
mappings. Due to the famous Koebe $1/4$-theorem, for any $f\in \mathcal S$, where $\mathcal S$
stands for the class of univalent holomorphic functions in $\mathbb D$
possessing the standard normalization conditions $f(0)=0$ and $f'(0)=1$, the
set $f(\mathbb D)$ contains the disk $\{w\in\mathbb C : |w|<1/4\}$, and the
radius $1/4$ is the best possible. For other classes of univalent mappings of
$\mathbb D$ preserving the origin (for instance, for the classes of univalent
harmonic mappings or univalent solutions of various elliptic
PDE), it is natural to pose the question to determine the best possible
radius of the disk centered at the origin, which is contained in the image of
$\mathbb D$ by functions from the class under consideration. It is reasonable to call this radius the Koebe radius for the corresponding
class of mappings. Thus, the Koebe radius for the class $\mathcal S$ is $1/4$ and this
value is attained by the Koebe function and its rotations
\begin{equation*}
K(z)=\frac{z}{(1-z)^2}.
\end{equation*}
The concept of the Koebe radius for some classes of harmonic mappings was
introduced without the special name by several authors, see, for example,
\cite{PCh22}.

Let us recall the definition of the class $\mathcal S_H^0$ of univalent
harmonic mappings of $\mathbb D$. It is well-known that any complex-valued
harmonic function $f$ in $\mathbb D$ has the form $f=h+\overline{g}$, where $h$ and
$g$ are holomorphic functions in $\mathbb D$ (these functions $h$ and $g$ are
called holomorphic and anti-holomorphic components of $f$, respectively). The
class $\mathcal S_H^0$ consists of all harmonic functions $f$ in $\mathbb D$
which are univalent in $\mathbb D$ and possess the normalization conditions
$g(0)=h(0)=g'(0)=0$ and $h'(0)=1$.

The question on determination the Koebe radius for the class $\mathcal S_H^0$ is
related with the coefficients estimate problem for this class that was posed
by Clunie and Sheil-Small \cite{ClSh84} and that was motivated and inspired
by the famous Bieberbach's problem on coefficients estimate for the class
$\mathcal S$, which was solved by de Branges \cite{dBr85}, and which states that
the Maclaurin coefficients $a_n$, $n=1,2,\ldots$, of a function $f\in \mathcal S$
satisfy the estimate $|a_n|\leqslant n$. This estimate, similar to the case
of Koebe radius for the class $\mathcal S$, is sharp and it is attained on the Koebe
function $K(z)$.

The problem and conjecture stated by Clunie and Sheil-Small were to prove the
following estimates for Maclaurin coefficients $a_n$ and $b_n$,
$n=1,2,\ldots$, of holomorphic and anti-holomorphic components of a function
$f\in\mathcal S_H^0$, respectively: $|a_n|\leqslant (n+1)(2n+1)/6$,
$|b_n|\leqslant (n-1)(2n-1)/6$ and $\big||a_n|-|b_n|\big|\leqslant n$, and to
show that all these estimates are sharp and are attained by the \emph{harmonic
Koebe function}
\begin{equation*}
K_H(z)=\frac{z-\frac{1}{2}z^2+\frac{1}{6}z^3}{(1-z)^3}+
\overline{\biggl(\frac{\frac{1}{2}z^2+\frac{1}{6}z^3}{(1-z)^3}\biggr)}.
\end{equation*}
This problem was solved for several subclasses of $\mathcal S_H^0$, see, for
example, \cite{ClSh84,PK15,PL25}, but for the class $\mathcal S_H^0$ itself
it remains open. Let us revert to the question about Koebe radius for the
class $\mathcal S_H^0$. The image of $\mathbb{D}$ under mapping by $K_H$
contains the disk $\{w : |w|<1/6\}$ and it does not contain any disk of
larger radius centered at the origin, so the Koebe radius for $\mathcal
S_H^0$ is not greater than $1/6$. Clunie and Sheil-Small \cite{ClSh84} proved that
this radius not less than $1/16$ and conjectured that the Koebe radius 
is exactly $1/6$. The main aim of this work is to prove this conjecture for a
sufficiently wide class of harmonic mappings. 

We state now the main result of the present paper. First of all we
need to define the subclass of $\mathcal S_H^0$ we are dealing with. One of
the principal characteristics of the harmonic mapping $f=h+\overline{g}$ is
its analytic dilatation
\begin{equation*}
w_f(z)=\frac{g'(z)}{h'(z)}.
\end{equation*}
Analytic dilatation satisfies the conditions of the Schwarz lemma if
$f\in\mathcal S_H^0$ (see \cite[p.~79]{D04}), so for any $f\in\mathcal S_H^0$ there exists such nonnegative
$k\leqslant1$ and $m\geqslant1$ that
\begin{equation*}
|w_f(z)|\leqslant k|z|^m,\quad |z|<1.
\end{equation*}
This estimate allows us to define, for a given numbers $k\leqslant1$ and
$m\geqslant1$, the class
\begin{equation*}
\mathcal S_H^0(k,m)=\{f\in\mathcal S_H^0 : |w_f(z)|\leqslant k|z|^m,\enspace z\in\mathbb D\}.
\end{equation*}
Our main result gives a new estimate from below of the Koebe radius for the
class $\mathcal S_H^0(k,m)$ (the estimate \eqref{newest} in
Theorem~\ref{th1}). Using this estimate we obtain new
estimate for the first nonzero Maclaurin coefficient of the holomorphic part
of function from the class $\mathcal S_H^0(k,m)$. We refer the reader to
\cite{AbAlP19} where the best known estimate is obtained for this
coefficient for the whole class $\mathcal S_H^0$.

Notice that we will not touch the problem about estimate of remaining
coefficients and refer the interested reader to \cite{PKSt17}, where it is
possible to find a lot of information about state of the art of this problem.

The paper is organized as follows. In Section~\ref{sec2} we formulate the main
result of the paper (Theorem~\ref{th1}) and several its corollaries, while in
Section~\ref{sec3} we present the proofs.

\section{Main results and their corollaries}\label{sec2}

The main result of this paper is the following theorem.
\begin{theorem}\label{th1}
Let $F\in\mathcal S_H^0(k,m)$. Then
\begin{equation}\label{newest}
|F(z)|\geqslant \frac{|z|}{4\,(1+k|z|^m)^{2/m}},\quad z\in\mathbb D.
\end{equation}
In particular, the set $F(\mathbb D)$ for any function $F$ from
$\mathcal S_H^0(k,m)$ contains the disk $\big\{w :
|w|<1/\big(4(1+k)^{2/m}\big)\big\}$.
\end{theorem}

Notice that the class $\mathcal S^0_H(1,1)$ coincides with the class $\mathcal S^0_H$,
so in this case theorem is equavalent to \cite[Theorem~4.4]{ClSh84} and the class
$\mathcal S^0_H(0,m)$ coincides with the class $\mathcal S$ and in this case we arrive to the
Koebe $1/4$--theorem (so the estimate is sharp). The proof is similar to one presented in~\cite{PCh22}. Moreover, the class
$\mathcal S_H^0(k,1)$ defined above coincides with the class $\mathcal
S_H^0\big((1+k)/(1-k)\big)$ introduced in~\cite{PCh22}, and Theorem~1.2
from~\cite{PCh22} may be treated as the special case of Theorem~\ref{th1}.

The following corollary immediately follows from Theorem~\ref{th1}.

\begin{corollary}\label{cor1}
Let $k\leqslant 1/2$ or $m\geqslant4$. Then the image of $\mathbb D$ under the mapping by
$f\in\mathcal S_H^0(k,m)$ contains the disk $\{w : |w|<1/6\}\subset\{w : |w|<1/\big(4(1+k)^{2/m}\big)\}$.
\end{corollary}

So, the conjecture of Clunie and Sheil-Small about 1/6 is proved for wide classes
of harmonic mappings.

When $m$ tends to infinity or $k$ tends to zero, the estimate in
Theorem~\ref{th1} tends to $1/4$, what is natural, because the module of $w_f(z)$
tends to zero, and so the class $\mathcal S^0_H(k,m)$ "tends" to the class $\mathcal S$, for which
the corresponding estimate is $1/4$ by the Koebe theorem.

Going further let $f\in\mathcal S_H^0$ and $f=h+\overline{g}$, where
\begin{equation*}
h(z)=z+\sum_{n=2}^\infty a_nz^n,\quad g(z)=\sum_{n=2}^\infty b_nz^n.
\end{equation*}
We consider the class $\mathcal S^{p,q}_H$ of harmonic mappings, for which
$p$ is the index of the first nonzero coefficient $a_p$, and $q$ is the index
of the first nonzero coefficient $b_q$. Then the following proposition is an
immediate consequence of Theorem~\ref{th1}:

\begin{proposition}\label{prop2}
The image of $\mathbb D$ under the mapping by a function $f\in\mathcal
S^{p,q}_H$ contains the disk $\{w : |w|<R_q\}$, where $R_q=2^{-2q/(q-1)}$. 
\end{proposition}

Moreover, for a function $f\in\mathcal S^{p,q}_H$ the estimate for the
coefficient $|a_p|$ may be obtained using the methods elaborated in
\cite{Sh90}.

\begin{proposition}\label{prop3}
Let $f\in\mathcal S^{p,q}_H$. Then
\begin{equation*}
|a_p|\leqslant d_q^{p-1}\Bigl(\frac{d_qR_q}{p}+p\Bigr),
\end{equation*}
where
\begin{equation*}
d_q=\frac{2\pi}{3\sqrt{3}R_q}.
\end{equation*}
Moreover,
\begin{equation*}
|a_2|\leqslant d_q\Bigl(\frac{d_qR_q}{2}+2\Bigr)-2.
\end{equation*}
\end{proposition}

When $q>5$, the latest estimate is lower
than 16.5 (which is the best known for the class $\mathcal S^0_H$ up to this moment \cite{AbAlP19}),
so we established new estimate of $a_2$ for the classes
$\mathcal S_{H}^{2,q}$ with $q>5$.   

One can apply Propositions~\ref{prop2} and~\ref{prop3} to the class
$\mathcal S_{H,\mathrm{odd}}^0$ of odd harmonic mappings. The class of odd
(holomorphic) univalent functions has been studied in details (see, for
example, \cite[\S~2.11]{D83}), but at the best of our knowledge odd
harmonic mappings are poor studied up to now. With respect to this class let
us state two results, which follow immediately from
Propositions~\ref{prop2} and~\ref{prop3} since $\mathcal
S_{H,\mathrm{odd}}^0\subset\mathcal S_H^{3,3}$.

\begin{corollary}\label{cor3}
The image of $\mathbb D$ under the mapping by a function $f\in\mathcal
S_{H,\mathrm{odd}}^0$ contains the disk $\{w : |w|<1/8\}$.
\end{corollary}

\begin{corollary}\label{cor4}
For every $f\in\mathcal S_{H,\mathrm{odd}}^0$ one has $|a_3|\leqslant 319$.
\end{corollary}

Since analytic Koebe function $K(z)$ lies in $\mathcal S^0_H(1,m)$ for every $m>1$, we can 
estimate Koebe radius for the classes from above with 1/4. 
For the class $\mathcal S^0_H(1,3)$ we can improve this upper bound.

Harmonic Koebe function $K_{H}(z)$ was constructed using the "shearing" construction
from the holomorphic Koebe function $K(z)$ with prescribed analytic dilatation $w(z)=z$ (see, for example, \cite{ClSh84}).
We can construct analogous functions $K_{H,m}(z)$ with analytic dilatation
$w(z)=z^m$. In the case $m=2,3,4$ we can find these functions explicitly (in the case
$m>4$ finding the explicit formula for $K_{H,m}(z)$ seems too difficult): 

\begin{gather*}
K_{H,2}(z)=-\frac{1}{3}-\frac{1}{3(z-1)^3}+\overline{-\frac{1}{3}+\frac{-1+3z-3z^2}{3(z-1)^3}}, \\
K_{H,3}(z)= \frac{1}{54}\left(-27+\frac{2 \pi}{\sqrt{3}}-\frac{3\left(9-7 z+2 z^2\right)}{(z-1)^3}-4 \sqrt{3}\arctan\left[\frac{1+2 z}{\sqrt{3}}\right]\right)+\\
+\overline{\frac{1}{54}\left(-27+\frac{2 \pi}{\sqrt{3}}-\frac{27-75 z+60 z^2+4 \sqrt{3}(z-1)^3 \arctan\left[\frac{1+2 z}{\sqrt{3}}\right]}{(z-1)^3}\right)},\\
K_{H,4}(z)=-\frac{2}{3}-\frac{8-9 z+3 z^2+3(z-1)^3 \arctan(z)}{12(z-1)^3}+\\
+\overline{-\frac{2}{3}-\frac{8-21 z+15 z^2+3(z-1)^3 \arctan(z)}{12(z-1)^3}}.
\end{gather*}

Direct calculation shows that $K_{H,2}(-1)=-1/3$, $K_{H,3}(-1)\approx-0.231289$ and $K_{H,4}(-1)\approx-0.273968$.
So, in case $m=3$ we can estimate Koebe radius (estimate from below follows from 
Theorem \ref{th1}) for the class $\mathcal S^0_H(1,3)$ in 
the following way: $0.157\leq R_{H,3}\leq 0.232$. Unfortunately, this method can't improve the upper
bound for the classes $\mathcal S^0_H(1,2)$ and $\mathcal S^0_H(1,4)$.

Finally, in the class $\mathcal S_H^0(k,m)$ yet another extremal problem can
be readily solved.
\begin{proposition}\label{th2}
Let $m\in\mathbb N$. For each function $f\in\mathcal S_H^0(k,m)$ we have
\begin{equation*}
\mathrm{Area}(f(\mathbb D))\geqslant \pi\biggl(1-\frac{k^2}{m+1}\biggr),
\end{equation*}
where $\mathrm{Area}$, as usual, stands for the 2-dimensional Lebesgue
measure. This inequality is sharp and the minimum is attained only at the
function $f(z)=z+\frac{k}{m+1}\overline{z}^{m+1}$ and its rotations.
\end{proposition}

\section{Proofs}\label{sec3}

The proof will appeal to the standard method of extremal length. Let us recall some
useful definitions.

Let $\Omega$ be a domain in the complex plane and let $\Gamma$ be a family of
locally rectifiable arcs $\gamma$ in $\Omega$. A metric is a Borel measurable
function $\rho(z)\geqslant0$ on $\Omega$. The $\rho$-length of an arc
$\gamma\in\Gamma$ is
\begin{equation*}
L(\gamma)=\int_{\gamma}\rho(z)|dz|.
\end{equation*}

A metric $\rho$ is said to be admissible for the curve family $\Gamma$ if
$L(\gamma)\geqslant1$ for every $\gamma\in\Gamma$. The extremal length of $\Gamma$
is the quantity $\lambda(\Gamma)$ defined by
\begin{equation*}
\frac{1}{\lambda(\Gamma)}=\inf\iint_{\Omega}\rho(z)^2dxdy, \quad z=x+iy,
\end{equation*}
where the infimum is taken over all admissible metrics $\rho$.

The module of a ring domain $\Omega$ is defined as
$\mu(\Omega)=\frac{1}{2\pi}\log\frac{b}{a}$ if $\Omega$ can be mapped
conformally onto an annulus $a<|z|<b$. The proof of the following proposition
can be found, for example, in~\cite[Appendix]{D04}.
\begin{proposition}\label{prop1}
The module of a doubly connected domain $\Omega$ coincides with the extremal
length of the family of curves joining the two boundary components of $\Omega$.
\end{proposition}

\begin{proof}[Proof of Theorem~\ref{th1}]
Consider the function $f(z):=\frac{F(az)}{a}$ for the function $F\in\mathcal S_H^0(k,m)$ and for some positive $a<1$ and let $\Omega=f(\mathbb{D})$ be the range of $f$. Then $f\in\mathcal S^{0}_{H}$ and $f$ is the homeomorphism of the closed disc $\overline{\mathbb D}$ to $\overline\Omega$. Let $\delta$ be the distance from the origin to the boundary curve $\partial\Omega$.

Fix a parameter $\beta>1$ and observe that $|f(z)|>\varepsilon$ on the circle $|z|=\beta\varepsilon$  for all sufficiently small $\varepsilon\in(0,\delta)$. Suppose without loss of generality that $\delta\in\partial\Omega$ (this can be achieved by rotating the function $f$).

Consider the following domains:
\begin{gather*}
\Omega_{\varepsilon}=\Omega\backslash\{w:|w|\leqslant\varepsilon\},\\
\Delta=\mathbb{C}\backslash\{w:-\infty<w\leqslant-\delta\} \\
\Delta_{\varepsilon}=\Delta\backslash\{w:|w|\leqslant\varepsilon\}.
\end{gather*}
The first step is to note the following inequality (proof may be found in
\cite[p.~93]{F67})
\begin{equation}\label{eq:1}
    \mu(\Omega_{\varepsilon})\leqslant\mu(\Delta_{\varepsilon}).    
\end{equation}

The next step is to estimate $\mu(\Omega_\varepsilon)$ from below. Let $\mathbb{D}_{\beta\varepsilon}$ be the ring $\{z:\beta\varepsilon <|z|<1\}$. Consider the metric $\rho$, admissible for a family of crosscuts in $\mathbb{D}_{\beta\varepsilon}$ and define $\rho(z)=0$ outside $\mathbb{D}_{\beta\varepsilon}$. Define a metric $\widetilde\rho$ on $\Omega_{\varepsilon}$ by
\begin{equation*}
\widetilde\rho(f(z))=\frac{\rho(z)}{|h^{\prime}(z)|-|g^{\prime}(z)|},
\end{equation*}
\noindent where $f=h+\overline g$. Let $\widetilde\gamma$ be a crosscut in $\Omega_{\varepsilon}$, and let $\gamma$ be the restriction of its preimage $f^{-1}(\widetilde\gamma)$ to a crosscut of $\mathbb{D}_{\beta\varepsilon}$. Then
\begin{equation*}
\int_{\widetilde\gamma}\widetilde\rho(z)|dz|=\int_{\gamma}\frac{\rho(z)|\frac{\partial f}{\partial s}(z)|}{|h^{\prime}(z)|-|g^{\prime}(z)|}|dz|,
\end{equation*}

\noindent where $|\partial f/\partial s(z)|$ denotes directional derivative of $f$ along $\gamma$.

But $|\partial f/\partial s(z)|\ge|h^{\prime}(z)|-|g^{\prime}(z)|$, so $\widetilde\rho$ is an admissible metric for the family of crosscuts in $\Omega_{\varepsilon}$.

Therefore,
\begin{equation*}
\frac{1}{\mu(\Omega_{\varepsilon})}\leqslant\iint_{\Omega_{\varepsilon}}\widetilde\rho(z)^2dxdy=\iint_{\mathbb{D}_{\beta\varepsilon}}\frac{\rho(z)^2J_f(z)}{(|h^{\prime}(z)|-|g^{\prime}(z)|)^2}dxdy,
\end{equation*}

\noindent where $J_f(z)=|h^{\prime}(z)|^2-|g^{\prime}(z)|^2$ is the Jacobian of $f$. But since $|w_f(z)|=|w_F(az)|\le ka^m|z|^m$, we get an estimate:
\begin{equation*}
\frac{J_f(z)}{(|h^{\prime}(z)|-|g^{\prime}(z)|)^2}=\frac{|h^{\prime}(z)|+|g^{\prime}(z)|}{|h^{\prime}(z)|-|g^{\prime}(z)|}=\frac{1+|w_f(z)|}{1-|w_f(z)|}\leqslant\frac{1+ka^m|z|^m}{1-ka^m|z|^m}
\end{equation*}

Thus, the inequality
\begin{equation}\label{eq:2}
\frac{1}{\mu(\Omega_{\varepsilon})}\leqslant\iint_{\mathbb{D}_{\beta\varepsilon}}\frac{1+ka^m|z|^m}{1-ka^m|z|^m}\rho(z)^2dxdy.
\end{equation}

\noindent holds for any metric $\rho$ that is admissible for the family of crosscuts in $\mathbb{D}_{\beta\varepsilon}$.

Now choose the admissible metric
\begin{equation*}
\rho(z)=\frac{1-ka^mr^m}{1+ka^mr^m}\frac{1}{r}\Bigl\{\int_{\beta\varepsilon}^{1}\frac{1-ka^mt^m}{1+ka^mt^m}\frac{1}{t}dt\Bigr\}^{-1}
\end{equation*}

Then, using the formula \ref{eq:2}, we obtain that
\begin{equation*}
\frac{1}{\mu(\Omega_{\varepsilon})}\leqslant 2\pi \Bigl\{\int_{\beta\varepsilon}^{1}\frac{1-ka^mt^m}{1+ka^mt^m}\frac{1}{t}dt\Bigr\}^{-1}
\end{equation*}

Straightforward integration leads to the following expression
\begin{equation*}
\int_{\beta\varepsilon}^{1}\frac{1-ka^mt^m}{1+ka^mt^m}\frac{1}{t}dt=-\frac{2}{m}\ln(1+ka^mt^m)\big\vert^{t=1}_{t=\beta\varepsilon}+\ln(t)\big\vert^{t=1}_{t=\beta\varepsilon}=\frac{2}{m}\ln\frac{1+ka^m\beta^m\varepsilon^m}{1+ka^m}-\ln\beta\varepsilon
\end{equation*}

So,
\begin{equation}\label{eq:3}
\mu(\Omega_{\varepsilon})\geqslant\frac{1}{2\pi}\Bigl(\frac{2}{m}\ln\frac{1+ka^m\beta^m\varepsilon^m}{1+ka^m}-\ln\beta\varepsilon\Bigr)
\end{equation}
On the other hand,  the function $\psi(z)=4\delta K(z)$ maps $\mathbb{D}$ onto $\Delta$, and so $\psi(\mathbb{D}_{\varepsilon(4\delta)^{-1}})$ agrees approximately with $\Delta_\varepsilon$ when $\varepsilon$ is small. Since the module is a conformal invariant, this gives the approximate formula
\begin{equation}\label{eq:4}
\mu(\Delta_{\varepsilon})=\frac{1}{2\pi}\ln\frac{4\delta}{\varepsilon}+\mathcal{O}(\varepsilon), \varepsilon\to0.
\end{equation}
Comparing \ref{eq:1}, \ref{eq:3} and \ref{eq:4}, we arrive at the following inequality
\begin{equation*}
\frac{2}{m}\ln\frac{1+ka^m\beta^m\varepsilon^m}{1+ka^m}-\ln\beta\varepsilon\leqslant \ln\frac{4\delta}{\varepsilon}+\mathcal{O}(\varepsilon).
\end{equation*}
Letting $\varepsilon$ tend to zero, and $\beta$ tend to 1, we deduce finally that
\begin{equation*}
\delta\geqslant\frac{1}{4(1+ka^m)^{\frac{2}{m}}}.
\end{equation*}

However, $|f(z)|\geqslant\delta$ when $|z|=1$. Thus, since $f(z)=\frac{F(az)}{a}$ and $a$ was chosen arbitrarily, it follows that
\begin{equation*}
|F(z)|\geqslant\frac{|z|}{4(1+k|z|^m)^{\frac{2}{m}}}.
\end{equation*}

for every $z\in\mathbb{D}$.
\end{proof}

\begin{proof}[Proof of Proposition~\ref{prop2}]
If $f\in \mathcal S^{p,q}_H$, then the order of the zero at the origin of $g^{\prime}(z)$ is equal to $q-1$. But the order of the zero at the origin of the analytic dilatation $w_f(z)$ coincides with the order of the zero of the function $g^{\prime}(z)$. So, due to the Schwarz lemma $|w_f(z)|\leqslant|z|^{q-1}$. Thus, $f$ belongs to the class $\mathcal S_H^0(1,q-1)$ and Theorem~\ref{th1} gives the proof of the proposition.
\end{proof}

\begin{proof}[Proof of Proposition~\ref{prop3}]
Let $\Delta$ be a preimage of the disk $|z|<R_q$ under $f$. Let $\phi$ be the conformal mapping of the unit disk onto $\Delta$ with the condition $\phi(0)=0,\phi'(0)>0$. Then $F(z)=R_q^{-1}f(\phi(z))$ is a harmonic mapping of the unit disk onto itself due to the Proposition~\ref{prop2}. On the other hand, if $\phi(z)=c_1z+c_2z^2+...$, then
\begin{equation*}
f(\phi(z))=\phi(z)+a_p\phi(z)^p+..+\overline{b_q\phi(z)^q}+..=c_1z+c_2z^2+...+(c_p+c_1^pa_p)z^p+..
\end{equation*}

Thus, with the notation $F(z)=\sum\limits_{n=1}^\infty A_nz^n+\overline{\sum\limits_{n=1}^\infty B_nz^n}$, we arrive at the formulas
\begin{equation*}
A_1=R_q^{-1}c_1,A_p=R_q^{-1}(c_p+c_1^pa_p)
\end{equation*}

Since $F$ maps the unit disk harmonically onto itself, it follows
\cite[\S~4.5]{D04} that $|A_n|\leqslant\frac{1}{n}$, or
\begin{equation*}
|c_p+c_1^pa_p|\leqslant\frac{R_q}{p}.
\end{equation*}

Thus, we have
\begin{equation}\label{eq:5}
|a_p|\leqslant c_1^{-p}\Bigl(\frac{R_q}{p}+|c_p|\Bigr).
\end{equation}

de Branges\cite{dBr85} theorem states , that
\begin{equation}\label{eq:6}
|c_p|\leqslant c_1p.
\end{equation}

Also, the exact form of the Heinz inequality, which is proved in \cite{H82},
gives the estimate  $|A_1|^2\geqslant\frac{27}{4}\pi^{-2}$, or
\begin{equation}\label{eq:7}
c_1\geqslant\frac{3\sqrt3}{2}R_q\pi^{-1}=d^{-1}_q.
\end{equation}

Combining \ref{eq:5}, \ref{eq:6} and \ref{eq:7}, we arrive at the estimate, stated in the proposition.

If $p=2$, then the estimate could be improved by using Pick's inequality for
bounded functions \cite[p.~74, ex.~33]{D83} instead of de Branges'
theorem. In this case we have
\begin{equation*}
|c_2|\leqslant 2c_1(1-c_1)
\end{equation*}

\noindent and with \ref{eq:5} and \ref{eq:7} the second part of the proposition is proved.
\end{proof}
\begin{proof}[Proof of Proposition~\ref{th2}]
The area $A$ of $f(\mathbb{D})$ is
\begin{gather*}
A =\iint_{\mathbb{D}}\left|h^{\prime}(z)\right|^2-\left|g^{\prime}(z)\right|^2 d x d y =\iint_{\mathbb{D}}\left(1-|w(z)|^2\right)\left|h^{\prime}(z)\right|^2 d x d y \geqslant\\
\geqslant \iint_{\mathbb{D}}\left(1-|kz^m|^2\right)\left|h^{\prime}(z)\right|^2 d x d y =\pi \sum_{n=1}^{\infty} n\left|a_n\right|^2-\iint_{\mathbb{D}}\left|\sum_{n=1}^{\infty}k n a_n z^{n+m-1}\right|^2 d x d y = \\
=\pi \sum_{n=1}^{\infty} n\left(1-k^2\frac{n}{n+m}\right)\left|a_n\right|^2=\pi\Bigl(1-\frac{k^2}{m+1}\Bigr)+\pi \sum_{n=2}^{\infty} n\left(1-\frac{k^2n}{n+m}\right)\left|a_n\right|^2,
\end{gather*}
\noindent as shown by a standard calculation in polar coordinates. The last sum is clearly minimized by choosing $a_n=0$ for all $n \geqslant 2$. The minimum area is attained only if $|w(z)| \equiv k|z|^m$, or $w(z)=e^{i \alpha} kz^m$. Hence, $g^{\prime}(z)=e^{i \alpha} kz^{m}$, which implies $\left|b_{m+1}\right|=\frac{k}{m+1}$ and $b_n=0$ for all $n \neq m+1$. But the function $f(z)=z+e^{i\alpha}\frac{k}{m+1} \bar{z}^{m+1}$ is univalent, so the theorem is proved.
\end{proof}

\end{document}